\documentclass[11pt]{article}
\usepackage[a4paper,margin=1in]{geometry}
\usepackage{amsmath,amssymb,amsthm,mathtools}
\usepackage{microtype}
\usepackage{hyperref}
\hypersetup{colorlinks=true,linkcolor=black,citecolor=black,urlcolor=black}

\newtheorem{theorem}{Theorem}
\newtheorem{lemma}{Lemma}

\title{An Optimal Stability Theorem for H\"older's Inequality}
\author{Gangsong Leng\\
East China Normal University\\
\texttt{lenggangsong@163.com}
\and
Yifan Lu\\
Zhejiang Wenzhou High School\\
\texttt{luyifan091120@outlook.com}}
\date{}

\begin{document}

\maketitle

\begin{abstract}
We prove an optimal $L^1$ stability theorem for H\"older's inequality. Let $p>1$, $q>1$, and $1/p+1/q=1$. If $a_k,b_k\ge 0$ and
\[
   \sum_{k=1}^n a_k=\sum_{k=1}^n b_k=1,
\]
then
\[
   1-\sum_{k=1}^n a_k^{1/p}b_k^{1/q}
   \ge \frac1{2pq}\left(\sum_{k=1}^n |a_k-b_k|\right)^2 .
\]
The constant $1/(2pq)$ is best possible. We also give the corresponding integral form.
\end{abstract}



\vspace{0.8em}

H\"older's inequality is a fundamental result in analysis and in the theory of inequalities. It gives an upper bound for sums or integrals of products, and it also describes the proportionality condition under which equality holds. Many classical inequalities have corresponding stability problems: if an inequality is nearly sharp, can one quantitatively show that the objects involved are close to the equality case? Aldaz studied stability for H\"older's inequality directly \cite{Aldaz2008} and related it to refinements of the AM--GM inequality \cite{Aldaz2009}. Carlen, Frank and Lieb, in their work on stability estimates for the lowest eigenvalue of a Schr\"odinger operator, also proved a H\"older-type stability estimate \cite{CFL2014}. Steele's book on inequalities also establishes a stability version of H\"older's inequality, but with a non-optimal constant \cite{Steele2004}. In this paper we discuss a natural discrete probabilistic version. Its form is simple, its constant is sharp, and its proof is completely elementary.

The classical form of H\"older's inequality states that, for $p>1$, $q>1$, and
\[
   \frac1p+\frac1q=1,
\]
one has
\[
   \sum_{k=1}^n x_k y_k
   \le
   \left(\sum_{k=1}^n x_k^p\right)^{1/p}
   \left(\sum_{k=1}^n y_k^q\right)^{1/q} .
\]
Putting
\[
   a_k=x_k^p,\qquad b_k=y_k^q,
\]
and imposing the normalization
\[
   \sum_{k=1}^n a_k=\sum_{k=1}^n b_k=1,
\]
we obtain the equivalent form
\[
   \sum_{k=1}^n a_k^{1/p}b_k^{1/q}\le 1 .
\]
The equality case corresponds to $a_k=b_k$ for $k=1,\ldots,n$. It is therefore natural to ask whether, when two probability vectors $a=(a_1,\ldots,a_n)$ and $b=(b_1,\ldots,b_n)$ deviate from the equality case, the deficit
\[
   1-\sum_{k=1}^n a_k^{1/p}b_k^{1/q}
\]
can be bounded from below in terms of the distance between $a$ and $b$. The following theorem gives an optimal answer.

\begin{theorem}[Optimal $L^1$ stability of H\"older's inequality]
Let $p>1$, $q>1$, and
\[
   \frac1p+\frac1q=1 .
\]
If $a_k,b_k\ge0$ and
\[
   \sum_{k=1}^n a_k=\sum_{k=1}^n b_k=1,
\]
then
\[
   1-\sum_{k=1}^n a_k^{1/p}b_k^{1/q}
   \ge
   \frac1{2pq}\left(\sum_{k=1}^n |a_k-b_k|\right)^2 . \tag{1}
\]
The constant $1/(2pq)$ is best possible.
\end{theorem}

Before proving Theorem 1, we first establish a strengthened form of Young's inequality. This is the key point of the proof.

\begin{lemma}[A strengthened Young inequality]
Let $p>1$, $q>1$, and
\[
   \frac1p+\frac1q=1 .
\]
Set
\[
   A=\frac{2-\frac1p}{3},\qquad
   B=\frac{1+\frac1p}{3} .
\]
Then $A+B=1$, and for all $a,b>0$,
\[
   \frac ap+\frac bq-a^{1/p}b^{1/q}
   \ge
   \frac1{2pq}\cdot \frac{(a-b)^2}{Aa+Bb}. \tag{2}
\]
\end{lemma}

\begin{proof}
By homogeneity, it is enough to take $b=1$ and write $a=x$. Thus we only need to prove
\[
   \frac xp+\frac1q-x^{1/p}
   \ge
   \frac1{2pq}\cdot \frac{(x-1)^2}{Ax+B},\qquad x>0 .
\]
Equivalently, it is enough to show that
\[
   F(x):=2pq(Ax+B)\left(\frac xp+\frac1q-x^{1/p}\right)-(x-1)^2\ge0 .
\]
A direct computation gives
\[
   F(1)=0,
   \qquad
   F'(1)=0 .
\]
A further differentiation and simplification yield
\[
   F''(x)=
   \frac{2(p+1)}{3p(p-1)}x^{1/p-2}
   \left[px^{2-1/p}-(2p-1)x+(p-1)\right]. \tag{3}
\]
We now prove that the expression in brackets is nonnegative. By the weighted AM--GM inequality,
\[
   \frac{p}{2p-1}x^{2-1/p}
   +\frac{p-1}{2p-1}\cdot 1
   \ge
   \left(x^{2-1/p}\right)^{p/(2p-1)}=x .
\]
Multiplying both sides by $2p-1$, we get
\[
   px^{2-1/p}+(p-1)\ge(2p-1)x .
\]
It follows from (3) that
\[
   F''(x)\ge0\qquad (x>0).
\]
Thus $F$ is convex. Since $F(1)=F'(1)=0$, the point $x=1$ is a global minimum of $F$, and hence $F(x)\ge0$. The lemma follows.
\end{proof}

We now prove Theorem 1. First assume that all $a_k,b_k$ are positive. By the lemma, for each pair $a_k,b_k$ we have
\[
   \frac{a_k}{p}+\frac{b_k}{q}-a_k^{1/p}b_k^{1/q}
   \ge
   \frac1{2pq}\cdot \frac{(a_k-b_k)^2}{Aa_k+Bb_k}.
\]
Summing over $k$ gives
\[
\begin{aligned}
   1-\sum_{k=1}^n a_k^{1/p}b_k^{1/q}
   &=\sum_{k=1}^n
   \left(\frac{a_k}{p}+\frac{b_k}{q}-a_k^{1/p}b_k^{1/q}\right) \\
   &\ge
   \frac1{2pq}\sum_{k=1}^n
   \frac{(a_k-b_k)^2}{Aa_k+Bb_k} .
\end{aligned}
\]
By the Cauchy--Schwarz inequality,
\[
   \sum_{k=1}^n
   \frac{(a_k-b_k)^2}{Aa_k+Bb_k}
   \ge
   \frac{\left(\sum_{k=1}^n |a_k-b_k|\right)^2}
        {\sum_{k=1}^n(Aa_k+Bb_k)} .
\]
Moreover,
\[
   \sum_{k=1}^n(Aa_k+Bb_k)
   =A\sum_{k=1}^n a_k+B\sum_{k=1}^n b_k
   =A+B=1.
\]
Therefore
\[
   \sum_{k=1}^n
   \frac{(a_k-b_k)^2}{Aa_k+Bb_k}
   \ge
   \left(\sum_{k=1}^n |a_k-b_k|\right)^2 .
\]
Consequently,
\[
   1-\sum_{k=1}^n a_k^{1/p}b_k^{1/q}
   \ge
   \frac1{2pq}\left(\sum_{k=1}^n |a_k-b_k|\right)^2 .
\]
If some of the $a_k$ or $b_k$ are zero, the result follows by a limiting argument. For instance, let
\[
   a_k^{(\varepsilon)}=\frac{a_k+\varepsilon}{1+n\varepsilon},
   \qquad
   b_k^{(\varepsilon)}=\frac{b_k+\varepsilon}{1+n\varepsilon}.
\]
Applying the already proved positive case to $a_k^{(\varepsilon)}$ and $b_k^{(\varepsilon)}$, and then letting $\varepsilon\to0^+$, gives the nonnegative case.

It remains to prove that the constant $1/(2pq)$ cannot be increased. In fact, for every fixed integer $n\ge2$, one can construct two nonnegative $n$-tuples for which the corresponding quotient tends to $1/(2pq)$. Take
\[
   a_1=a_2=\frac12,
   \qquad
   a_3=\cdots=a_n=0,
\]
and
\[
   b_1=\frac{1+\varepsilon}{2},
   \qquad
   b_2=\frac{1-\varepsilon}{2},
   \qquad
   b_3=\cdots=b_n=0,
\]
where $0<\varepsilon<1$. Clearly
\[
   \sum_{k=1}^n a_k=\sum_{k=1}^n b_k=1 .
\]
Also,
\[
   \sum_{k=1}^n |a_k-b_k|
   =\left|\frac12-\frac{1+\varepsilon}{2}\right|
    +\left|\frac12-\frac{1-\varepsilon}{2}\right|
   =\varepsilon .
\]
On the other hand, since all terms from the third to the $n$-th vanish,
\[
\begin{aligned}
   \sum_{k=1}^n a_k^{1/p}b_k^{1/q}
   &=\left(\frac12\right)^{1/p}
     \left(\frac{1+\varepsilon}{2}\right)^{1/q}
     +\left(\frac12\right)^{1/p}
     \left(\frac{1-\varepsilon}{2}\right)^{1/q}       \\
   &=\frac12(1+\varepsilon)^{1/q}
     +\frac12(1-\varepsilon)^{1/q} .
\end{aligned}
\]
Let $\alpha=1/q$. As $\varepsilon\to0$, the binomial expansion gives
\[
   (1+\varepsilon)^\alpha
   =1+\alpha\varepsilon+\frac{\alpha(\alpha-1)}2\varepsilon^2+O(\varepsilon^3),
\]
and
\[
   (1-\varepsilon)^\alpha
   =1-\alpha\varepsilon+\frac{\alpha(\alpha-1)}2\varepsilon^2+O(\varepsilon^3).
\]
After adding the two identities, the linear terms cancel, and hence
\[
   \frac12(1+\varepsilon)^\alpha+\frac12(1-\varepsilon)^\alpha
   =1+\frac{\alpha(\alpha-1)}2\varepsilon^2+O(\varepsilon^3).
\]
Since
\[
   \alpha=\frac1q,
   \qquad
   \alpha-1=-\frac1p,
\]
we have
\[
   \alpha(\alpha-1)=-\frac1{pq} .
\]
Therefore
\[
   \sum_{k=1}^n a_k^{1/p}b_k^{1/q}
   =1-\frac1{2pq}\varepsilon^2+O(\varepsilon^3),
\]
and so
\[
   1-\sum_{k=1}^n a_k^{1/p}b_k^{1/q}
   =\frac1{2pq}\varepsilon^2+O(\varepsilon^3).
\]
Since
\[
   \left(\sum_{k=1}^n |a_k-b_k|\right)^2=\varepsilon^2,
\]
it follows that
\[
   \frac{1-\sum_{k=1}^n a_k^{1/p}b_k^{1/q}}
        {\left(\sum_{k=1}^n |a_k-b_k|\right)^2}
   =\frac1{2pq}+O(\varepsilon).
\]
Letting $\varepsilon\to0$, we obtain
\[
   \lim_{\varepsilon\to0}
   \frac{1-\sum_{k=1}^n a_k^{1/p}b_k^{1/q}}
        {\left(\sum_{k=1}^n |a_k-b_k|\right)^2}
   =\frac1{2pq} .
\]
Thus, for each fixed $n\ge2$, the constant in Theorem 1 cannot exceed $1/(2pq)$. Combined with the inequality already proved, this shows that $1/(2pq)$ is the best possible constant.

\begin{theorem}[Integral form]
Let $(X,\mathcal A,\mu)$ be a measure space, and let $p>1$, $q>1$ with $1/p+1/q=1$. If $f,g\ge0$ are integrable functions satisfying
\[
   \int_X f\,d\mu=\int_X g\,d\mu=1,
\]
then
\[
   1-\int_X f^{1/p}g^{1/q}\,d\mu
   \ge
   \frac1{2pq}
   \left(\int_X |f-g|\,d\mu\right)^2 . \tag{4}
\]
\end{theorem}

This is the integral version of the stability theorem for H\"older's inequality. Its proof is completely parallel to the discrete case. Applying the same strengthened Young inequality almost everywhere on $X$, we have
\[
   \frac{f(x)}p+\frac{g(x)}q-f(x)^{1/p}g(x)^{1/q}
   \ge
   \frac1{2pq}\cdot
   \frac{(f(x)-g(x))^2}{Af(x)+Bg(x)}, \tag{5}
\]
where
\[
   A=\frac{2-\frac1p}{3},\qquad
   B=\frac{1+\frac1p}{3}.
\]
At points where $f(x)=g(x)=0$, the right-hand side is interpreted as $0$. Integrating (5), we get
\[
   1-\int_X f^{1/p}g^{1/q}\,d\mu
   \ge
   \frac1{2pq}
   \int_X \frac{(f-g)^2}{Af+Bg}\,d\mu .
\]
By the integral form of the Cauchy--Schwarz inequality,
\[
   \int_X \frac{(f-g)^2}{Af+Bg}\,d\mu
   \ge
   \frac{\left(\int_X |f-g|\,d\mu\right)^2}
        {\int_X (Af+Bg)\,d\mu} .
\]
Moreover,
\[
   \int_X (Af+Bg)\,d\mu
   =A\int_X f\,d\mu+B\int_X g\,d\mu
   =A+B=1.
\]
This gives (4). Thus the case of discrete probability vectors is a special case of the integral stability estimate.

\end{document}